\documentclass[11pt,leqno]{article}

\usepackage{amsmath, amsfonts, amsthm, 
}
\usepackage{latexsym}
\usepackage{graphicx}
\usepackage{color}
\makeatletter
\def\@maketitle{\newpage
    \null
    \vskip .8truein
    \begin{center}%
     {\bf \@title \par}%
     \vskip 1.5em
     {\small
      \lineskip .5em
      \begin{tabular}[t]{c}\@author
      \end{tabular}\par}%
    \end{center}%
    \par
    \vskip .4truein}

\let\nn=\nonumber

\newcommand{\re}{{\mathbb R}}
\newcommand{\nat}{{\mathbb N}}
\newcommand{\tor}{{\mathbb T}^N}
\newcommand{\Gh}{\widehat G}
\newcommand{\Fh}{\widehat F}

\newcommand{\z}{{{I\!\!Z}}}

\newtheorem{theorem}{Theorem}[section]
\newtheorem{lemma}[theorem]{Lemma}
\newtheorem{proposition}[theorem]{Proposition}
\newtheorem{corollary}[theorem]{Corollary}
\theoremstyle{remark}
\newtheorem{remark}[theorem]{Remark}

\theoremstyle{definition}

  {\hfill$\Box$\bigskip\par}
\DeclareMathOperator{\argmax}{argmax}

\def\proof{\list{}{\setlength{\leftmargin}{0pt}
                      \parskip=0pt\parsep=0pt\listparindent=2em
                      \itemindent=0pt}\item[]\futurelet\testchar\@maybe}

\def\@maybe{\ifx[\testchar \let\next\@Opt
          \else \let\next\@NoOpt \fi \next}
\def\@Opt[#1]{{\it Proof of #1.\ }}\def\@NoOpt{{\it Proof.\ }}
\begin{document}
\title{\Large \bf Short time existence for a general backward-forward parabolic system arising from Mean-Field Games}

\author{{\large \sc Marco Cirant\thanks{Dipartimento di Matematica ``Tullio Levi-Civita'', Universit\`a di Padova, cirant@math.unipd.it}, Roberto Gianni\thanks{Dipartimento di Matematica e Informatica ``U. Dini'', Universit\`a di Firenze, roberto.gianni@unifi.it}, Paola Mannucci\thanks{Dipartimento di Matematica ``Tullio Levi-Civita'', Universit\`a di Padova, mannucci@math.unipd.it}}}  


\maketitle


\begin{abstract}
We study the local in time existence of a regular solution of a nonlinear parabolic backward-forward system arising  
from the theory of Mean-Field Games (briefly MFG).  The proof is based on a contraction argument in a suitable space that takes account of the peculiar structure of the system, which involves also a coupling at the final horizon. We apply the result to obtain existence to very general MFG models, including also congestion problems.

\end{abstract}



\section{Introduction}
Let $\mathbb{T}^N=\re^N/\z^N$ be the $N$-dimensional flat torus. Denote by $Q_T= \mathbb{T}^N\times (0,T)$.
We consider the following nonlinear backward-forward parabolic system:
\begin{equation}\label{P}
\begin{cases}
-u_t-a_{ij}(x,t)u_{x_ix_j}+F(u, m, Du, Dm, x,t)=0, & \text{in $Q_T$} \\
m_t-c_{ij}(x,t)m_{x_ix_j}+G(u, m, Du, Dm, D^2u, x,t)=0, & \text{in $Q_T$} \\
u(x,T)=h[m(T)](x),\quad m(x,0)=m_0(x), & \text{in $\mathbb{T}^N$},
\end{cases}
\end{equation}
where $h$ is a regularising nonlocal term.\\
The aim of this paper is to study the short time existence of a regular solution of system \eqref{P} under very general assumptions on the data. 
The peculiarities of the system are: 1) nonlinear backward-forward parabolic form; 2) the final condition on $u$ depends on $m$ through a regularising nonlocal term; 3) the coupling functions $F$ and $G$ can have a very general form, but $F$ does not depend on the second derivatives of the unknowns.\\
From the structure 1) and 2), classical results on forward parabolic systems cannot be directly applied and the problem of well posedness is non standard.
The general structure 1)-3) of \eqref{P} is inspired by parabolic systems arising from the theory of Mean-Field Games (briefly MFG), where $u$ represents the value function of a stochastic control problem and $m$ is a density distribution of a population of identical players. In a typical MFG setting, the functions $u$ and $m$ satisfy the following system of two equations (called  Hamilton-Jacobi-Bellman and Fokker-Plank, respectively): 
\begin{equation}\label{MFG0}
\begin{cases}
-u_t - A_{ij} u_{ij} + H(x, t, Du, m) = 0 & \text{in $Q_T$} \\
m_t - \partial_{ij} (A_{ij} m) - {\rm div} (m D_p H(x, t, Du, m)) = 0, & \text{in $Q_T$} \\
u(x,T)=h[m(T)](x),\quad m(x,0)=m_0(x) & \text{in $\mathbb{T}^N$},
\end{cases}
\end{equation}
where $A(x,t) = \frac{1}{2}\Sigma \Sigma^T(x,t)$ and the Hamiltonian $H$ is the Legendre transform of some Lagrangian function $L$, i.e.
\[
H(x, t, p, m) = \sup_{v \in \re^N} \{p\cdot v - L (x, t, v, m)\}.
\]
We refer to Section \ref{MFG} for a more detailed derivation of this system. \smallskip

As for the general problem \eqref{P}, under the assumptions stated at the beginning of the following section, the main existence theorem can be stated as follows:
\begin{theorem}\label{mainTH}
Under the assumptions (A1)-(A5) there exists $\overline T>0$ such that for all $T \in (0, \overline T]$ the problem \eqref{P} has a solution $u,m\in W^{2,1}_p(Q_T)$ with $p>N+2$ satisfying 
equations in \eqref{P} a.e..
\end{theorem}
\noindent The solution found in Theorem \ref{mainTH} is locally unique in the sense specified in Remark \ref{localun}.
The proof of the theorem is based on a contraction procedure in a suitable space, that takes into account the forward-backward structure of the system which has a coupling also at the final horizon. We only require $F$ and $G$ to be bounded with respect to $x, t$ and {\it locally} Lipschitz continuous with respect to the other entries; in addition, $G$ is required to be {\it globally} Lipschitz continuous with respect to the entry of the second order term $D^2 u$. This is a natural assumption for the models that we have in mind (see in particular the equation for $m$ in \eqref{MFG0}).  As stated in point 2), $h$ should be a regularising function of $m$. Such gain of regularity is true for example when one considers $h$ of convolution form, or $h$ independent of $m$. The gain of regularity of $h$ is crucial in our fixed point method. Without this assumption the argument would need additional smallness of other data. For additional comments, see Remark \ref{remass} and Section \ref{counter}, where it is shown that existence for arbitrary small times $T$ may even fail for linear problems when $h$ is not regularizing.

Our existence result can be applied to very general MFG models. The existence of smooth solutions for systems of the form \eqref{MFG0} has been explored in several works, see e.g. \cite{Car, CT, GPS1, GPS2, Gbook, LLe, LL2} and references therein. Existence for arbitrarily large time horizons $T$ typically requires assumptions on the behaviour of $H$ at infinity, that are crucial to obtain a priori estimates. Our result is for short-time horizons, but just requires enough {\it local} regularity of $H$: we have basically no restrictions on the behaviour of $H$ when its entries are large. We are in particular interested in MFG models with congestion, that are particularly delicate due to the presence of a singular Hamiltonian $H$. Short time existence has been discussed in \cite{GV}, \cite{Gr} under suitable growth assumptions on $H$, relying on the peculiar MFG structure.
For a detailed description and derivation of MFG systems, additional references and the statements of our results on congestion problems, see Section \ref{MFG}. 

The paper is organized as follows: in Section 2 we state the assumptions and we present some preliminary results. In Section 3 we give the proof of the main theorem. We also give a counterexample for a very simple linear system where the final condition is of local (non-regularizing) type.
In Section 4 we apply the result to prove short time existence of a solution to some general classes of MFGs.
In the Appendix we give the proof of the classical estimate in the periodic setting, stated in Section 2, that is used extensively.

\medskip

\noindent {\em Notations}: For any non-negative real number $r \ge 0$ and $q \ge 1$, we will denote by $W_q^{r}(\tor)$ the (fractional) Sobolev-Slobodeckij space of periodic functions (see \cite[p. 70]{LSU} for its definition); we will denote by $\|u\|^{(m)}_{q, \tor}$ its norm. Note that when $r \in \nat$, $W_q^{r}(\tor)$ is a standard Sobolev space. For any positive integer $m$, $W_q^{2m, m}(Q_T)$, with norm $\|u\|^{(2m)}_{q,Q_T}$, will be the usual Sobolev parabolic space (see \cite[p. 5]{LSU}). $W^{1,0}_p (Q_T)$, with norm $\|u\|^{(1)}_{q,Q_T}$, will be the space of functions in $L^p(Q_T)$ with weak derivatives in the $x$-variable in $L^p(Q_T)$. 
For any real and non-integer number $r >0$, $C^{r,r/2}(Q_T)$ with norm $|u|^{(r)}_{Q_T}$ will be the standard H\"older parabolic space (see \cite[p. 7]{LSU}, where the alternative notation $H^{r,r/2}$ is used). Note that here we mean that the regularity is up to the parabolic boundary, hence, since the spatial variable varies in the torus, up to  $t=0$.
Finally, $C^{1,0}(Q_T)$ with norm $|u|^{(1)}_{Q_T}$ will be the space of continuous functions on $Q_T$ with continuous derivatives in the $x$-variable, up to $t=0$ as for the H\"older spaces. We denote by $\|\cdot \|_{\infty}$ the $\infty$-norm. $S^N$ denotes the space of symmetric matrices of order $N$.

\section{Setting of the problem and preliminary results}\label{ass}
In this section we state our standing assumptions and we write some useful lemmata and propositions. 
Throughout the paper we assume:
\begin{itemize}
\item[(A1)] $a_{ij}(x,t)$ and $c_{ij}(x,t)$ are continuous functions on $Q_T$.
\item[(A2)] $F(a,b,p,q, x,t): \re\times\re^+\times\re^N\times\re^N\times Q_T \rightarrow \re$ is 
such that
for all $M > 0$ there exists $L_F(M) > 0$ ($L_F(M)$ is an increasing function of $M$, bounded for bounded values of $M$) such that
\begin{eqnarray*}
&&| F(a_1,b_1,p_1,q_1, x,t)|_{\infty}\leq L_F(M),\\
&&| F(a_1,b_1,p_1,q_1, x,t)-F(a_2,b_2,p_2,q_2, x,t)| \leq \\ 
&&L_F(M)(|a_1-a_2|+|b_1-b_2|+|p_1-p_2|+|q_1-q_2|),
\end{eqnarray*} 
for all $|a_i|, |b_i|, |b_i|^{-1}, |p_i|, |q_i| \le 2M$, $i = 1,2$ and all $(x,t) \in Q_T$.
\item[(A3)] $G(a,b,p,q,H,x,t): \re\times\re^+\times\re^N\times\re^N\times S^N\times Q_T \rightarrow \re$ is such that for all $M > 0$ there exists $L_G(M) > 0$ 
($L_G(M)$ is an increasing function of $M$, bounded for bounded values of $M$) such that
\begin{eqnarray*}
&&| G(a_1,b_1,p_1,q_1, H_1, x,t)|_{\infty}\leq L_G(M)(1+|H_1|),\\
&&| G(a_1,b_1,p_1,q_1,H_1,x,t)-G(a_2,b_2,p_2,q_2,H_2,x,t)| \leq\\ 
&& L_G(M)(|a_1-a_2|+|b_1-b_2|+|p_1-p_2|+|q_1-q_2|)(1 + |H_1|)+\\
&&L_G(M)|H_1-H_2|,
\end{eqnarray*}
for all $|a_i|, |b_i|, |b_i|^{-1}, |p_i|, |q_i| \le 2M$, $i = 1,2$ and all $H_i \in S^N$, $(x,t) \in Q_T$.
\item[(A4)] $h: C^{1}(\mathbb{T}^N) \to C^{2}(\mathbb{T}^N)$, and there exists $L_h > 0$ such that
$|h[m_1]- h[m_2]|_{\tor}^{(2)}\leq L_h|m_1-m_2|_{\tor}^{(1)}$.
\item[(A5)] $m_0\in W^2_{\infty}(\mathbb{T}^N)$ and $m_0\geq \delta>0$.
\end{itemize}

%
%
Before we prove the theorem, some remarks on the assumptions and useful preliminary lemmata are in order.

\begin{remark}\label{remass} First, note that (A2) and (A3) require $F$ and $G$ to be bounded with respect to $x, t$ and {\it locally} Lipschitz continuous with respect to $a,b,p,q$. Note that for $G$ we need a linear dependence on $H$, this is a natural assumption for the models we have in mind.
Moreover, $G$ is required to be {\it globally} Lipschitz continuous with respect to $H$, that corresponds to the entry of the second order term $D^2 u$.

By (A4), $h$ should be a {\it regularizing} function of $m$. Such gain of regularity holds for example when one considers $h$ of the form $h[m] = h_0(m \star \psi)$, where $h_0$ is a twice differentiable function and $\psi$ is a smoothing kernel. Another example is to consider a constant function of $m$, namely $h[m] = u_T$, where $u_T \in C^{2}(\mathbb{T}^N)$. The gain of regularity of $h$ is crucial in our fixed point method. Without this assumption, say if $h[m](x) = h_0(m(x))$, the argument would need additional smallness of other data. In this case, as we will see in Section \ref{counter}, existence for arbitrary small times $T$ may even fail for linear problems.
\end{remark}
\begin{lemma}\label{hbound}
There exists $C_0 > 0$ such that
\begin{equation}\label{A4con}
|h[m]|_{\tor}^{(2)}\leq L_h|m|_{\tor}^{(1)}+C_0.
\end{equation}
\end{lemma}
\begin{proof}
Since 
$$|h[m_1]- h[m_0]|_{\tor}^{(2)}\geq |h[m_1]|_{\tor}^{(2)}- |h[m_0]|_{\tor}^{(2)},$$
hence from (A4) and (A5)
$$|h[m_1]|_{\tor}^{(2)}\leq  L_h|m_1-m_0|_{\tor}^{(1)}+
 |h[m_0]|_{\tor}^{(2)}\leq  L_h|m_1|_{\tor}^{(1)}+L_h|m_0|_{\tor}^{(1)}+
 |h[m_0]|_{\tor}^{(2)}.$$
\end{proof}

\begin{lemma}\label{L1}
Let $\alpha \in (0, 1)$. For any $f\in C^{1+\alpha, (1+\alpha)/2}(Q_T)$,
 \begin{equation}\label{holdgen}
  |f|^{(1)}_{Q_T}\leq |f(\cdot,0)|^{(1)}_{\tor}+ T^{\alpha/2}|f|^{(1+\alpha)}_{Q_T}.
 \end{equation}
 \end{lemma}
 \begin{proof} Follows immediately from the definition of $|f|^{(1+\alpha)}_{Q_T}$. \end{proof}
 
\begin{lemma}\label{L2}
Let $q \ge 2$ and $f \in L^q(Q_T)$ be such that
$
 \| f\|_{q, Q_T}\leq C.
$
 Then, for $p=q/2$,
  \begin{equation}\label{L22prima}
 \|f\|_{p, Q_T}\leq C\, T^{\frac{1}{2p}}.
 \end{equation}
Let $f\in W^{2,1}_{q}(Q_T)$ be such that
$
 \| f\|^{(2)}_{q, Q_T}\leq C.
$
 Then, for $p=q/2$
  \begin{equation}\label{L22}
 \|f\|^{(2)}_{p, Q_T}\leq C\, T^{\frac{1}{2p}}.
 \end{equation}
  \end{lemma}
  
 \begin{proof}
 We prove \eqref{L22prima}, \eqref{L22} is analogous.
 By H\"older inequality applied to $|f|^p$ for any $r > 1$, take $r'$ such that $\frac{1}{r}+\frac{1}{r'}=1$ we have
 \begin{multline*}
\|f\|_{p,Q_T}=\left(\int_{Q_T}|f|^pdxdt\right)^{1/p}\leq 
\left (\int_{Q_T}|f|^{pr}dxdt\right)^{1/pr}\left(\int_{Q_T}dxdt\right)^{1/pr'}=\\
\|f\|_{pr,Q_T}\left(\int_{Q_T}dxdt\right)^{(r-1)/pr}=\|f\|_{pr,Q_T}(|\tor|T)^{(r-1)/pr}\leq 
 C\|f\|_{pr,Q_T}T^{(r-1)/pr}.
 \end{multline*}
 Choosing $r$ such that $q=rp$, we have
 \begin{equation}\label{beta}
 \|f\|_{p, Q_T}\leq C\,T^{\frac{q-p}{pq}},\ with\ q>p.
 \end{equation}
Taking $r=2$, i.e. $q=2p$
 we have the result.
 \end{proof}

We recall now the following embedding proposition proved by R. Gianni in \cite{RG}. Observe that the constant $M$ remains bounded for bounded values of $T$ while in the estimate  of Corollary of p.342  of \cite{LSU} it blows up as $T$ tends to zero.
 \begin{proposition}[\it Inequality (2.21) of \cite{RG}]\label{gianni}
 Let $f\in W^{2,1}_{q}(Q_T)$. Then,
 \begin{equation}\label{embRob}
  |f|^{(2-\frac{N+2}{p})}_{Q_T}\leq M\left(\|f\|^{(2)}_{p, Q_T}+ \|f(x,0)\|^{(2-2/p)}_{p, \tor}\right), \qquad p>\frac{N+2}{2},\ p\neq N+2,
 \end{equation}
where $M$ remains bounded for bounded values of $T$.
\end{proposition}

 In the following proposition we state some regularity results for linear parabolic equations in the flat torus. Such results are classical and well-known for equations on cylinders with boundary conditions (see \cite{LSU}); for the convenience of the reader, we show that they hold true also for equations that are set in the domain $Q_T=\mathbb{T}^N\times (0,T)$, and basically follow from local parabolic regularity. Let 
\begin{equation}\label{linapp}
\begin{cases}
{\mathcal L}u:= u_t-\sum_{ij}a_{ij}(x,t)u_{x_ix_j}+ \sum_ia_i(x,t)u_{x_i}+ a(x,t)u=f(x,t) & \text{in $Q_T$}, \\
u(x,0)=u_0(x) & \text{in $\mathbb{T}^N$.}
\end{cases}
\end{equation}
H1) Suppose that the functions $a_{ij}$, $a_i$, $a$, $f$ belong to $C^{\alpha,\alpha/2}(Q_T)$ and $u_0(x)\in C^{2+\alpha}(\mathbb{T}^N)$.\\
H2) Suppose that the functions $a_{ij}$, $a_i$, $a$ are continuous functions in $Q_T$, $f\in L^q(Q_T)$  and 
$u_0(x)\in W_q^{2-2/q}(\mathbb{T}^N)$ with $q>3/2$.\\
\begin{proposition}\label{w2pest}
Under assumptions H1) there exists a unique solution $u\in C^{2+\alpha,1+\alpha/2}(Q_T)$ of problem \eqref{linapp}
and the following estimate holds:
\begin{equation}\label{320}
  |u|^{(\alpha+2)}_{Q_T}\leq C_1\left(|f|^{(\alpha)}_{Q_T}+ 
  |u_0|^{(\alpha+2)}_{\tor}\right).
 \end{equation}
 where the constant $C_1$ depends only on the norms of the coefficients $a_{ij}$, $a_i$, $a$ specified in H1), on $N, \alpha$ and $T$, and remains bounded for bounded values of $T$.\\
Under assumptions H2) there exists a unique solution $u\in W_q^{2,1}(Q_T)$ of problem \eqref{linapp}
and the following estimate holds:
 \begin{equation}\label{342}
  \|u\|^{(2)}_{q,Q_T}\leq C_2\left(|f|_{q,Q_T}+ 
  |u_0|^{(2-2/q)}_{q, \tor}\right),
 \end{equation}
where the constant $C_2$ depends only on the norms of the coefficients $a_{ij}$, $a_i$, $a$ specified in H2), on $N, q$ and $T$, and remains bounded for bounded values of $T$.
\end{proposition}
\begin{proof}
The proof is given in Appendix \ref{appe}.
\end{proof}

\section{The existence theorem}\label{ex}

In this section we prove Theorem  \ref{mainTH}.  At the end of the section we give a simple counterexample where existence may fail.

\begin{proof}[Theorem \ref{mainTH}]
{\bf Step 1: Lipschitz regularization of $F, G$}. Let $K > 0$ be large enough, so that 
\begin{equation}\label{Kl}
K \ge \max\left\{2|m_0|^{(1)}_{\tor}, 2\left(L_h|m_0|^{(1)}_{\tor}+ C_0\right), \frac{2}{\delta}\right\},
\end{equation}
where $\delta, L_h, C_0$ are as in (A4), (A5) and \eqref{A4con}. Let $\varphi, \bar \varphi : \re \to \re$ be globally Lipschitz functions such that $\varphi(x) = x$ for all $x \in [1/K, K]$, $ \varphi(x) \in [1/(2K), 2K]$ for all $x \in \re$, and $\bar \varphi(x) = x$ for all $x \in [-K, K]$, $\bar \varphi(x) \in [-2K, 2K]$ for all $x \in \re$. Similarly, let $\psi : \reì^N \to \re^N$ be a globally Lipschitz function such that $\psi(p) = p$ for all $|p| \le K$ and $|\psi(p)| \le 2K$ for all $p \in \re^N$.

We will construct a solution to \eqref{P} with $F, G$ replaced by $\widehat F, \widehat G$ defined as follows:
\[
\begin{split}
\widehat F(a,b,p,q, x,t) &= F(\bar \varphi(a), \varphi(b),\psi(p),\psi(q), x,t) \\
\widehat G(a,b,p,q,H,x,t) &= G(\bar \varphi(a), \varphi(b),\psi(p),\psi(q),H, x,t).
\end{split}
\]
Note that by (A3), $\Gh$ satisfies
\begin{multline}\label{Ghl}
| \Gh(a_1,b_1,p_1,q_,H_1,x,t)-\Gh(a_2,b_2,p_2,q_2,H_2,x,t)| \leq \\ L_G(|a_1-a_2|+|b_1-b_2|+|p_1-p_2|+|q_1-q_2|)(1 + |H_1|)+L_G|H_1-H_2|
\end{multline}
for all $a_i, b_i, p_i, q_i, H_i, x, t$ (possibly by a constant $L_G$ that is larger than the one in (A3)). Moreover, again by (A3) and the fact that $|\bar \varphi|, |\varphi|, |\psi|$ are bounded by $2K$, we have for some $L(K) > 0$
\begin{equation}\label{Ghb}
|\Gh(a,b,p,q,H,x,t)| = |G(a,b,p,q,H,x,t)| \le L(K)(|H| + 1)
\end{equation}
for all $a,b,p,q,H,x,t$. Analogous bounds hold also for $\Fh$ by (A2). 

{\bf Step 2: fixed point set-up}. Let us define the space (see also Remark \ref{rkX})
\begin{multline*}
X_M^T=\{(u,m): u\in W^{2,1}_p (Q_T)\cap C^{1,0}(Q_T),
m\in C^{1,0}(Q_T), \\
\|u\|^{(2)}_{p, Q_T}+|u|^{(1)}_{Q_T}+ |m|^{(1)}_{Q_T} \leq M,\ p>N+2\}.
\end{multline*}
Define now the operator $\mathcal{T}$ on $X_M^T$ in the following way:\\
$$\mathcal{T}(\hat u, \hat m)=(\overline u, \overline m)$$ where $(\overline u, \overline m)$ is the solution 
of the following problems
\begin{equation}\label{PP1}
\begin{cases}
\overline m_t-c_{ij}(x,t)\overline m_{x_ix_j}+\Gh(\hat m,\hat u,D\hat u, D\hat m, D^2\hat u, x,t)=0,& in\ Q_T \\
\overline m(x,0)=m_0(x), & in\ \mathbb{T}^N.\end{cases}
\end{equation}
\begin{equation}\label{PP2}
\begin{cases}
 -\overline u_t-a_{ij}(x,t)\overline u_{x_ix_j}+\Fh(\overline m,\hat u,D\hat u, D\overline m, x,t)=0,& in\  Q_T\\
\overline u(x,T)=h(\overline m(x,T),x),& in\ \mathbb{T}^N
\end{cases}
\end{equation}
We aim at showing that $\mathcal{T}$ is a contraction on $X_M^T$ for suitable $M$ and small $T$.

{\bf Step 3: $\mathcal{T}$ maps $X_M^T$ into itself}, that is, $\mathcal{T}(\hat u, \hat m)=(\overline u, \overline m)\in X_M^T$ for any $(\hat u, \hat m)\in X_M^T$.
Denote by $\overline u^*(x, T-t):=\overline u(x,t)$. The couple of functions $(\overline m, \overline u^*)$ solves \eqref{PP1} and
\begin{equation}\label{PP2main}
\begin{cases}
\overline u^*_t-a_{ij}(x,T-t)\overline u^*_{x_ix_j}+ \\
\widehat F(\overline m(x,T-t),\hat u(x,T-t),D\hat u(x,T-t), D\overline m(x,T-t), x,T-t)=0,& in\  Q_T,\\
 \overline u^*(x,0)=h(\overline m(T,x),x), & in\ \mathbb{T}^N.
 \end{cases}
\end{equation}
The initial condition for $\overline u^*(x,0)$ depends on $\overline m(T,x)$ which is well defined from the regularity of $\overline m(t,x)$ obtained below.
Note that problem \eqref{PP1} is well posed, namely there exists a unique solution $\overline m(x, t)$ such that
$\overline m \in W^{2,1}_{p}(Q_T)$ (see \cite[Theorem 9.1 p. 341]{LSU}).
Moreover, since $\Gh$ satisfies \eqref{Ghb} and $\|\hat u\|^{(2)}_{p, Q_T}\leq M$, then 
 the term $\Gh(\hat m,\hat u,D\hat u, D\hat m, D^2\hat u, x,t)$ is such that  $\|\Gh\|_{p, Q_T}\leq L(K)(M + 1)$. We can therefore apply 
Proposition \ref{w2pest} to \eqref{PP1}  to get
  \begin{equation}\label{prima}
 \|\overline m\|^{(2)}_{p, Q_T}\leq C\left(L(K)(M + 1) +  \|m_0\|^{(2-2/p)}_{p, \tor}\right)\leq C(M),
 \end{equation}
(in what follows, we will not make the dependence on constants on $K$ explicit). Hence, by the embedding \eqref{embRob}, we have the following inequality:
 \begin{equation}\label{embRobe}
  |\overline m|^{(2-\frac{N+2}{p})}_{Q_T}\leq C\left(\|\overline m\|^{(2)}_{p, Q_T}+ \|m_0(x)\|^{(2-2/p)}_{p, \tor}\right), \qquad p>\frac{N+2}{2},\ p\neq N+2,
 \end{equation}
 where $C$ is bounded for bounded values of $T$.\\
 Hence, from \eqref{prima} and (A5) 
 \begin{equation}\label{est}
 |\overline m|^{(2-\frac{N+2}{p})}_{Q_T}\leq C(M),\ p>\frac{N+2}{2},\ p\neq n+2.
 \end{equation}
 Since $p>N+2$, i.e. $2-\frac{N+2}{p}>1$, then \eqref{holdgen} 
easily yields
\begin{equation}\label{contr1}
 |\overline m|^{(1)}_{Q_T}\leq |m_0|^{(1)}_{\tor}+ T^{\frac{1}{2}-\frac{n+2}{2p}}
C(M).
 \end{equation}
 In particular, note that, from \eqref{est}, we have that the trace $\overline m(x,T)$ is well defined
 and
 \begin{equation}\label{tracem}
 |\overline m(x,T)|^{(1)}_{\tor}\leq C(M).
 \end{equation}
 
 We now pass to study the well posedness of problem  \eqref{PP2main} and the regularity of its solution $\overline u^*$.
  From estimate \eqref{tracem}, the regularising assumptions (A4) and \eqref{A4con} on $h$, the initial condition 
  $\overline u^*(x,0)=h(\overline m(T,x),x)$ is well defined.
  In turn, when $\overline m(x,t)$ is assigned with the regularity found above (see \eqref{est}) problem 
 \eqref{PP2main} admits a solution $\overline u^*$ by boundedness of $\Fh$. 
   From the initial condition for $\overline u^*$ and \eqref{A4con},
  \begin{equation}\label{initialu}
 |\overline u^*(x,0)|^{(2)}_{\tor}\leq L_h|\overline m(x,T)|^{(1)}_{\tor}+C_0.
 \end{equation}
By \eqref{contr1},
   \begin{equation}\label{initialudue}
 |\overline u^*(x,0)|^{(2)}_{\tor}\leq L_h|m_0|^{(1)}_{\tor}+ L_hC(M)\,T^{\frac{1}{2}-\frac{n+2}{2p}}
+C_0\leq C(M).
 \end{equation}
In particular, taking into account again that $\tor$ is bounded, for any $q > 1$, for some constant $C$ we have 
  \begin{equation}\label{initialutre}
 \|\overline u^*(x,0)\|^{(2-2/q)}_{q,Q_T}\leq C\|\overline u^*(x,0)\|^{(2)}_{q,Q_T}\leq C|\overline u^*(x,0)|^{(2)}_{Q_T}\leq C(M).
 \end{equation}
 We now study the regularity of $\overline u^*$.
Since the estimate in Proposition \ref{w2pest} is valid for any $q$, we obtain, because of the boundedness of $\Fh$, \eqref{contr1} and \eqref{initialutre},
 \begin{equation}\label{alpha}
 \|\overline u^*\|^{(2)}_{q, Q_T}\leq C(M, q), \text{ for any } q.
 \end{equation}
%
Applying \eqref{L22} of Lemma \ref{L2} we get
  \begin{equation}\label{beta2}
 \|\overline u^*\|^{(2)}_{p, Q_T}\leq C(M, 2p)\, T^{\frac{1}{2p}}.
 \end{equation}
 Hence, using again embedding \eqref{embRob} and \eqref{beta2} we obtain
 \begin{equation}\label{embRobu}
  |\overline u^*|^{(2-\frac{n+2}{p})}_{Q_T}\leq C\left(\|\overline u^*\|^{(2)}_{p, Q_T}+ \|\overline u^*(x,0)\|^{(2-2/p)}_{p, Q_T}\right)\leq C(M, p),  
 \end{equation}
Therefore, using \eqref{holdgen} of Lemma \ref{L1} and taking into account  \eqref{embRobu}, we have
   \begin{equation}\label{beta3}
 |\overline u^*|^{(1)}_{Q_T}\leq |\overline u^*(x,0)|^{(1)}_{Q_T}+ C(M,p)T^{\frac{1}{2}-\frac{n+2}{2p}}.
 \end{equation}
At this point, using estimate \eqref{initialudue} we obtain
 \begin{equation}\label{beta4}
|\overline u^*|^{(1)}_{Q_T}\leq L_h|m_0|^{(1)}_{\tor}+ C_0+C(M,p)T^{\frac{1}{2}-\frac{n+2}{2p}}.
 \end{equation}
Now we can easily see that \eqref{beta2} and \eqref{beta4} together with \eqref{contr1} allow us to prove that 
 $\mathcal{T}$ maps $X_M^T$ into itself. Indeed,
 \begin{multline*}
\|\overline u^*\|^{(2)}_{p, Q_T}+ |\overline u^*|^{(1)}_{Q_T}+
|\overline m|^{(1)}_{Q_T}\\
\leq C(M, 2p)\, T^{\frac{1}{2p}}+  
(L_h+1)|m_0|^{(1)}_{\tor}+ C_0+C_1(M,p)T^{\frac{1}{2}-\frac{n+2}{2p}}.
\end{multline*}
At this point 
we choose $$M_1:=3((L_h+1)|m_0|^{(1)}_{\tor} + C_0)$$ and we take $T$ sufficiently small that
$$C(M_1, 2p)\, T^{\frac{1}{2p}}\leq M_1/3$$ and
$$C_1(M_1, p)T^{\frac{1}{2p}-\frac{n+2}{p}}\leq  M_1/3,$$ thus obtaining 
$$\|\overline u^*\|^{(2)}_{p, Q_T}+ |\overline u^*|^{(1)}_{Q_T}+|\overline m|^{(1)}_{Q_T}\leq M_1,$$ that is,
$$\mathcal{T}: X_{M_1}^T\to X_{M_1}^T,$$
for all $T$ sufficiently small.

 {\bf Step 4:} $$\mathcal{T}: X_{M_1}^T\to X_{M_1}^T$$ is a {\bf contraction} operator.\\
Let $(\hat u_i,\hat m_i)\in X_{M_1}^T$, $i=1,2$. Let us denote 
$\mathcal{T}(\hat u_2,\hat m_2)=:(\overline u_2,\overline m_2)$ and 
 $\mathcal{T}(\hat u_1,\hat m_1)=:(\overline u_1,\overline m_1)$.
 We have to prove that 
 \begin{multline*}
\|\overline u_1-\overline u_2\|^{(2)}_{p, Q_T}+ |\overline u_1-\overline u_2|^{(1)}_{Q_T}+|\overline m_1-\overline m_2|^{(1)}_{Q_T}\leq\\
 \gamma
 \bigg(\|\hat u_1-\hat u_2\|^{(2)}_{p, Q_T}+ |\hat u_1-\hat u_2|^{(1)}_{Q_T}+
 |\hat m_1-\hat m_2|^{(1)}_{Q_T}\bigg),\nn
 \end{multline*}
 with $0<\gamma<1$.\\
 We denote by $\overline U:=\overline u_1-\overline u_2$, $\overline M:=\overline m_1-\overline m_2$,
 $\hat U:=\hat u_1-\hat u_2$, and $\hat M:=\hat m_1-\hat m_2$.
 Denoting by $\overline U^*(x,T-t):=\overline U(x,t)$,
 taking into account \eqref{PP1}-\eqref{PP2} and \eqref{PP2main},
  $\overline U^*$ and  $\overline M$ satisfy:
  
\begin{equation}\label{PP1contr}
\begin{cases}
\overline M_t-c_{ij}(x,t)\overline M_{x_ix_j}+\Gh(\hat u_1,\hat m_1,D\hat u_1, D\hat m_1, D^2\hat u_1, x,t)-  \\
\quad \Gh(\hat u_2,\hat m_2,D\hat u_2, D\hat m_2, D^2\hat u_2, x,t)=0, & \text{ in $Q_T$}, \\
\overline M(x,0)=0, & \text{in $\mathbb{T}^N$}.\end{cases}
\end{equation}
\begin{equation}\label{PP2contr}
\begin{cases}
\overline U^*_t-a_{ij}(x,t)\overline U^*_{x_ix_j}+&\\
\quad \Fh(\overline u_1(x,T-t),\hat m_1(x,T-t),D\hat u_1(x,T-t), D\overline m_1(x,T-t), x,T-t)-&\\
\quad \Fh(\overline u_2(x,T-t),\hat m_2(x,T-t),D\hat u_2(x,T-t), D\overline m_2(x,T-t), x,T-t)=0,\\ 
 \overline U^*(x,0)=h(\overline m_1(T,x),x)-h(\overline m_2(T,x),x), \qquad \text{ in $ \mathbb{T}^N$}.\end{cases}
\end{equation}
 Since $\hat u_i$ and $\hat m_i$, $i=1,2$ belong to $X_{M_1}^T$, we follow the same procedure as in Step 1.
 First, note that $\Gh$ satisfies \eqref{Ghl}, so
 \begin{multline}\label{so}
 \| \Gh(\hat u_1,\hat m_1,D\hat u_1, D\hat m_1, D^2\hat u_1)- \Gh(\hat u_2,\hat m_2,D\hat u_2, D\hat m_2, D^2\hat u_2)\|_{p, Q_T} \le \\
 L_G\left( |\hat U|^{(1)}_{Q_T}+|\hat M|^{(1)}_{Q_T}\right) \|1 + |D^2\hat u_1| \|_{p, Q_T}+L_G\|\hat U\|^{(2)}_{p, Q_T} \\ \le C \left( |\hat U|^{(1)}_{Q_T}+ \|\hat U\|^{(2)}_{p, Q_T} +|\hat M|^{(1)}_{Q_T} \right)
 \end{multline}
 Therefore, by Proposition \ref{w2pest} and $\overline M(x,0)=0$,
 \begin{equation}\label{contrM}
 \|\overline M\|^{(2)}_{p, Q_T}\leq C(|\hat U|^{(1)}_{Q_T}+ \|\hat U\|^{(2)}_{p, Q_T} +|\hat M|^{(1)}_{Q_T}  ).
 \end{equation}
Hence, from \eqref{contrM}, \eqref{holdgen} and the embedding \eqref{embRob}, 
 \begin{equation}\label{M2}
 |\overline M|^{(1)}_{Q_T}\leq
C(p)(|\hat U|^{(1)}_{Q_T}+ \|\hat U\|^{(2)}_{p, Q_T} +|\hat M|^{(1)}_{Q_T})T^{\frac{1}{2}-\frac{n+2}{2p}}.
 \end{equation}
 As far as $\overline U^*$ is concerned, from Proposition \ref{w2pest},
   \begin{equation}\label{contrU0}
 \|\overline U^*\|^{(2)}_{q, Q_T}\leq C(q)(\|\hat U\|^{(1)}_{q, Q_T}+\|\overline M\|^{(1)}_{q, Q_T}+
 \|\overline U^*(x,0)\|^{(2-2/q)}_{q, \tor}) \text{ for any } q.
 \end{equation}
 Note that, from assumption (A4) and using the boundedness of $\tor$, for some constant $C$ we have
 \begin{multline}\label{stimaU0}
\|\overline U^*(x,0)\|^{(2-2/q)}_{q, \tor}\leq C\|\overline U^*(x,0)\|^{(2)}_{q, \tor}=
 C \|h[\overline m_1(T)]-h[\overline m_2(T)]\|^{(2)}_{q, \tor}\\ 
\leq C|h[\overline m_1(T)]-h[\overline m_2(T)]|^{(2)}_{\tor}\leq CL_h|\overline M(T,x)|^{(1)}_{\tor}.
  \end{multline}
  Hence from \eqref{stimaU0}, \eqref{contrU0} becomes
  \begin{equation}\label{contrU}
 \|\overline U^*\|^{(2)}_{q, Q_T}\leq C(q)(\|\hat U\|^{(1)}_{q, Q_T}+\|\overline M\|^{(1)}_{q, Q_T}+
 |\overline M(T,x)|^{(1)}_{\tor}).
 \end{equation}
Then, in view of \eqref{M2} and boundedness of $Q_T$,
\begin{equation}\label{contrU2}
 \|\overline U^*\|^{(2)}_{q, Q_T}\leq C(q)(|\hat U|^{(1)}_{Q_T}+ \|\hat U\|^{(2)}_{p, Q_T} +|\hat M|^{(1)}_{Q_T}),  \text{ for any } q.
 \end{equation}
From \eqref{L22} of Lemma \ref{L2} we obtain
  \begin{equation}\label{contrU2}
 \|\overline U^*\|^{(2)}_{p, Q_T}\leq C(p)(|\hat U|^{(1)}_{Q_T}+ \|\hat U\|^{(2)}_{p, Q_T} +|\hat M|^{(1)}_{Q_T})T^{1/2p}.
 \end{equation}
 From the embedding result \eqref{embRob} (see also \eqref{beta3}), we have
  \begin{multline}\label{contrU3}
 |\overline U^*|^{(1)}_{Q_T}\leq  |\overline U^*|^{(2-\frac{N+2}{p})}_{Q_T}\le C_1(p)(\|\overline U^*\|^{(2)}_{p, Q_T} + \|\overline U^*(x,0)\|^{(2-2/p)}_{p, \tor}) \\
 \le C(p)(|\hat U|^{(1)}_{Q_T}+ \|\hat U\|^{(2)}_{p, Q_T} +|\hat M|^{(1)}_{Q_T})(T^{1/2p} + T^{\frac{1}{2}-\frac{n+2}{2p}}).
%
 \end{multline}
 where the last inequality comes from \eqref{contrU2}, \eqref{stimaU0} and \eqref{M2}.
 

 At this point, taking into account \eqref{contrU2}, \eqref{contrU3}, \eqref{M2}, for $T$ sufficiently small, we have proved that the operator $\mathcal{T}$ is a contraction. 
 The fixed point $(u^*(T-t),m)$ is a solution to \eqref{P} with $F, G$ replaced by $\widehat F, \widehat G$, with the required regularity.
 
 {\bf Step 5: back to the initial problem.} Note that $\Fh$ and $\Gh$ coincide with $F$ and $G$ respectively whenever the fixed point $(u,m )$ satisfies $m(x, t) \in [1/K, K]$, $u(x, t) \in [-K, K]$ and $|Du(x, t)|, |Dm(x,t)| \le K$ on $Q_T$. This is true if $T$ is sufficiently small. Indeed, by \eqref{beta4} and the choice \eqref{Kl} of $K$ one has
 \[
|u|^{(1)}_{Q_T} = |u^*|^{(1)}_{Q_T}\leq L_h|m_0|^{(1)}_{\tor}+ C_0+C(M,p)T^{\frac{1}{2}-\frac{n+2}{2p}} \le K,
 \]
 while by \eqref{contr1},
 \[
|m|^{(1)}_{Q_T}\leq |m_0|^{(1)}_{\tor}+ T^{\frac{1}{2}-\frac{n+2}{2p}}C(M) \le K.
 \]
 Finally, by \eqref{est} and (A4),
 \[
 \min_{Q_T} m \ge \min_{\tor} m(x, 0) -  |m|^{(2-\frac{N+2}{p})}_{Q_T} T^{\frac{1}{2}-\frac{n+2}{2p}} \ge \delta - C(M) T^{\frac{1}{2}-\frac{n+2}{2p}} \ge \frac{1}{K},
 \]
 that yields the desired result.
 \end{proof}
 \begin{remark}\label{localun}
 From \eqref{prima}, $m\in W^{2,1}_p(Q_T)$, with $p>N+2$.
Hence the function found above is locally unique in the following sense:
for any $M>0$ sufficiently large, there exists $T_M$ such that for any $T<T_M$ there exists an unique solution 
 of \eqref{P}, $(u,m)\in X_{M}^T$.
 \end{remark}
 
\begin{remark}
If we assume also that $a_{ij}$, $c_{ij}$, $F$ and $G$ are H\"older continuous with respect to $x,t$, if $m_0\in C^{2+\alpha}(\mathbb{T}^N)$ and $h$ takes its values in $C^{2+\alpha}(\mathbb{T}^N)$, then the solution of Theorem \ref{mainTH} will belong to
$C^{2+\alpha, (2+\alpha)/2 }(Q_T)$.
\end{remark}



\begin{remark}\label{rkX}
In the definition of the space $X_M^T$ we take $u$ belonging both to $W^{2,1}_p (Q_T)$ and $C^{1,0}(Q_T)$. This may appear unnecessary, since $W^{2,1}_p (Q_T)$ is continuously embedded in $C^{1,0}(Q_T)$. The crucial point is that such embedding depends on $T$; to rule out this dependence one has to make the initial datum explicit (see in particular \eqref{embRob}). Therefore, to simplify a bit the argument we preferred to control separately both $\|u\|^{(2)}_{p, Q_T}$ and $|u|^{(1)}_{Q_T}$.
\end{remark}

\subsection{Non-regularizing $h$: a counterexample to short-time existence}\label{counter}
As mentioned in Remark \ref{remass}, it is crucial in our fixed point method that $h$ in the final condition
$u(x,T)=h[m(T)]$ be a regularising function of $m$.
We will show in the sequel that without this assumption, existence for arbitrary small times $T$ may even fail for linear problems.
Let us consider the following linear parabolic backward-forward system, with $\alpha\in\re$ to be chosen
\begin{equation}\label{linear}
\begin{cases}
-u_t-\Delta u=0, & \text{in $\mathbb{T}^N\times (0,T)$}, \\
m_t-\Delta m=\Delta u, & \text{in $\mathbb{T}^N\times (0,T)$}, \\
u(x,T)=\alpha\, m(x,T),\quad m(x,0)=m_0(x) & \text{in $\mathbb{T}^N$}.
\end{cases}
\end{equation}
Here, $h[m](x) = \alpha m(x)$, and clearly $h[m]$ has the same regularity of $m$. Thus, $h$ does not satisfy (A4).
We claim that
\[
\begin{array}{c}
\textit{For all $\alpha < -2$, there exist smooth initial data $m_0$ and a sequence $T_k \to 0$} \\ 
\textit{such that \eqref{linear} is not solvable on $[0, T_k]$.}
\end{array}
\]

Suppose that, for some $T > 0$, there exists a  solution  $(u, m)$  to  \eqref{linear}.
 Let $\lambda_k$ and $\phi_k(x)$, $k\geq 0$ be the eigenvalues and eigenfunctions of the Laplace operator on $\tor$, i.e.
$$-\Delta\phi_k=\lambda_k\phi_k,\ \phi_k(x)\in C^{\infty}(\tor),\ k\geq 0.$$
Let $m_k(t)=\int_{\tor}m(x,t)\phi_k(x)dx$, $u_k(t)=\int_{\tor}u(x,t)\phi_k(x)dx$,\ $m_{0k}=\int_{\tor}m_0(x)\phi_k(x)dx$.
We can represent $(u, m)$ by
$$m(x,t)=\sum_0^{+\infty}m_k(t)\phi_k(x),\ u(x,t)=\sum_0^{+\infty}u_k(t)\phi_k(x),$$
and $m_k$ and $u_k$ satisfy
\begin{equation}\label{lineark}
\begin{cases}
-u_k'+\lambda_ku_k=0 & \text{in $(0,T)$}, \\
 m_k'+\lambda_km_k=-\lambda_ku_k & \text{in $(0,T)$}, \\
u_k(T)=\alpha\,m_k(T),\quad m_k(0)=m_{0k}.
\end{cases}
\end{equation}

We will suppose that the coefficients of the initial datum satisfy $m_{0k} \neq 0$ for all $k$ (this is possible as soon as $m_{0k}$ vanishes sufficiently fast as $k \to \infty$).

Deriving the second equation and taking into account the first one in \eqref{lineark}, we get
\begin{equation}\label{eqmk}
\begin{cases}
m_k''-\lambda_k^2m_k=0 & \text{in $(0,T)$}, \\
m_k(0)=m_{0k},   \\
\lambda_k\,(\alpha+1)\,m_k(T)=-m_k'(T).
\end{cases}
\end{equation}
Solving \eqref{eqmk} we obtain that 
\begin{equation}\label{mkeq}m_k(t)=A_k\sinh(\lambda_kt)+B_k\cosh(\lambda_k t)\end{equation} for some $A_k, B_k \in \re$, where
$$B_k=m_{0k} \neq 0.$$ If  $(\alpha+1)\sinh(\lambda_k T)+\cosh(\lambda_k T)\neq 0$, then $A_k$ is uniquely determined, i.e.
$$A_k= -B_k\frac{(\alpha+1)\cosh(\lambda_k T)+\sinh(\lambda_k T)}{(\alpha+1)\sinh(\lambda_k T)+\cosh(\lambda_k T)}.$$

Note that if  $\alpha<-2$, $(\alpha+1)\sinh(\lambda_k T)+\cosh(\lambda_k T)$ vanishes for positive values of $T$, and in particular when $T$ coincides with some 
$$T_k :=\frac{1}{\lambda_k}\tanh^{-1}\left(-\frac{1}{\alpha+1}\right).$$

In such case we reach a contradiction, since $(\alpha+1)\cosh(\lambda_k T)+\sinh(\lambda_k T) \neq 0$, and therefore $A_k$ cannot be determined. Since \eqref{mkeq} has no solutions, $m$ cannot exist. Finally, by the fact that $\lambda_k\to +\infty$ as $k\to +\infty$, we have $T_k \to 0$,
hence a short time existence result (as stated in Theorem \ref{mainTH}) cannot hold: for any $\overline T$ there exists a $T_k\in(0,\overline T]$
such that the $k-$th problem does not admit a solution in $[0, T_k]$, and for this reason, problem \eqref{linear} cannot be solved.

\begin{remark}
Note that  without the regularising assumption on $h$ the existence argument of Theorem \ref{mainTH} would work supposing additional smallness of some data.
For example, one could consider the equation $m_t-\Delta m=\epsilon\Delta u$ (in a system like \eqref{linear}) with $\epsilon$ sufficiently small,
or final datum $u(x,T)=\alpha\, m(x,T)$ with $|\alpha|$ and $m_0(x)$ suitably small. This is coherent with the previous non-existence counterexample where $\alpha<-2$.
\end{remark}

\section{Some parabolic systems arising in the theory of Mean-Field Games}\label{MFG}

Mean-Field Games (MFG) have been introduced simultaneously by Lasry and Lions \cite{LL1}, \cite{LL2}, \cite{LL3} and Huang et. al. \cite{HMC} to describe Nash equilibria in games with a very large number of identical agents. A general form of a MFG system can be derived as follows. Consider a given population density distribution $m(x, t)$. A typical agent in the game wants to minimize his own cost by controlling his state $X$, that is driven by a stochastic differential equation of the form
\begin{equation}\label{sde}
d X_s = -v_s ds + \Sigma(X_s, s) d B_s \qquad \forall s > 0,
\end{equation}
where $v_s$ is the control, $B_s$ is a Brownian motion and $\Sigma(\cdot, \cdot)$ is a positive matrix. The cost is given by
\[
\mathbb{E}^{X_0} \left[ \int_0^T L(X_s, s, v_s, m(X_s, s)) ds + h[m(T)](X_T) \right],
\]
where $L$ is some Lagrangian function, $h$ is the final cost, defined as a functional of $m(\cdot, T)$ and the state $X_T$ at the final horizon $T$ of the game. Assume that all the data are periodic in the $x$-variable. Formally, the dynamic programming principle leads to an Hamilton-Jacobi-Bellman equation for the value function of the agent $u(x,t) = \mathbb{E}^{x} \int_t^T L ds + h[m(T)](X_T)$, that is, $u$ solves
\begin{equation}\label{hjb}
\begin{cases}
-u_t - A_{ij} u_{ij} + H(x, t, Du, m) = 0 & \text{in $\mathbb{T}^N\times (0,T)$}, \\
u(x,T)=h[m(T)](x) & \text{in $\mathbb{T}^N$},
\end{cases}
\end{equation}
where $A(x,t) = \frac{1}{2}\Sigma \Sigma^T(x,t)$ and the Hamiltonian $H$ is the Legendre transform of $L$ with respect to the $v$ variable, i.e.
\[
H(x, t, p, m) = \sup_{v \in \re^N} \{p\cdot v - L (x, t, v, m)\}.
\]
Moreover, the optimal control $v^*_s$ of the agent is given in feedback form by
\[
v^*(x, s) \in \argmax_v \{Du(x, s) \cdot v - L (x, s, Du(x, s), m(x,s))\}
\]
Typically, one assumes $L$ to be convex in the $v$-entry. In this case, $H$ is strictly convex in the $p$-entry, and $v^*(x, s)$ can be uniquely determined by $$v^*(x, s) =D_p H (x, s, Du(x, s), m(x,s)).$$

In an equilibrium situation, since all agents are identical, the distribution of the population should coincide with the distribution of all the agents when they play optimally. Hence, the density of the law of every single agent should satisfy the
following Fokker-Planck equation
\begin{equation}\label{fkp}
\begin{cases}
m_t - \partial_{ij} (A_{ij} m) - {\rm div} (m D_p H(x, t, Du, m)) = 0 & \text{in $\mathbb{T}^N\times (0,T)$}, \\
m(x,0)=m_0(x) & \text{in $\mathbb{T}^N$},
\end{cases}
\end{equation}
where $m_0$ is the density of the initial distribution of the agents (\eqref{fkp} can be derived by plugging $v^*$ into \eqref{sde} and using the Ito's formula).

The coupled system of nonlinear parabolic PDEs \eqref{hjb}-\eqref{fkp} with backward-forward structure is of the form \eqref{P} if
\begin{equation}\label{FGmfg}
\begin{split}
F(x,t, u, m, Du, Dm)& = H(x, t, Du, m), \\
G(x,t, u, m, Du, D^2 u, Dm)& = -(\partial_{x_i x_j} A_{ij})m - 2 (\partial_{x_i} A_{ij}) m_{x_j} - D_p H \cdot D m \\
& -m(H_{x_i p_i} + H_{p_i p_j} u_{x_i x_j} + H_{mp_i} m_{x_i}),
\end{split}
\end{equation}
where the dependence on $H$, $A$, $m$ and their derivatives with respect to $(m, Du, x, t)$ and $(x, t)$ respectively has been omitted for brevity.

A general short-time existence result for \eqref{hjb}-\eqref{fkp} reads as follows.

\begin{theorem}\label{shortmfg} Suppose that $A \in C^{2}(Q_T)$, $h$ satisfies (A4), $m_0$ satisfies (A5) and
\begin{itemize}
\item $H$ is continuous with respect to $x,t,p,m$,
\item $H$, $\partial_{p_i} H$, $\partial_{x_i p_i}^2 H$, $\partial_{p_i p_j}^2 H$, $\partial_{m p_i}^2 H$ are locally Lipschitz continuous functions with respect to $p, m \in \re^N \times \re^+$, uniformly in $x, t \in Q_T$.
\end{itemize}
Then, there exists $\overline T>0$ such that for all $T \in (0, \overline T]$ the system \eqref{hjb}-\eqref{fkp} admits a regular solution $u,m\in W^{2,1}_q(Q_T)$ with $q>N+2$.
\end{theorem}
\begin{proof} In view of \eqref{FGmfg} and the standing assumptions, it suffices to apply Theorem \ref{mainTH}. \end{proof}

A typical case in the MFG literature is when $L$ has split dependence with respect to $m$ and $v$, that is,
\[
L(x, t, v, m) = L_0(x, t, v) + f(x, t, m).
\]
The existence of smooth solutions in this case has been explored in several works, see e.g. \cite{Car, CT, GPS1, GPS2, Gbook} and references therein. Existence for arbitrary time horizon $T$ typically requires assumptions on the behaviour of $H$ at infinity, that are crucial to obtain a priori estimates. As stated in the Introduction, our result is for short-time horizons, but no assumptions on the behaviour at infinity of $H$ are required. Note finally that $C^2$ regularity of $H$ is crucial for uniqueness in short-time, while for large $T$ uniqueness may fail in general even when $H$ is smooth (see \cite{BC, BF, Cpre, CT}).

\subsection{Congestion problems}

A class of MFG problems that attracted an increasing interest during the last few years is the so-called congestion case, namely when
\[
L(x, t, v, m) = m^\alpha L_1(v) + f(x, t, m),
\]
where $\alpha > 0$ and $L_1$ is a convex function. The term $m^\alpha$ penalizes $L_1(v)$ when $m$ is large, so agents prefer to move at low speed in congested areas. On the other hand, as soon as the environment density $m$ approaches zero, an agent can increase his own velocity without increasing significantly his cost. The parameter $\alpha$ can be then regarded as the strength of congestion. The difficulties in this problem are mainly caused by the singular term $m^\alpha$. It has been firstly discussed by Lions \cite{LLe}, and has been subsequently addressed in a series of papers. In \cite{GE, GNP, GM} the stationary case is treated. As for the time-dependent problem, short-time existence of weak solutions, under some restrictions on $\alpha$ and $H$, has been proved in \cite{Gr}. A general result of existence of weak solutions for arbitrary time horizon $T$ is discussed in \cite{AP}. So far, smoothness of solutions has been verified in the short-time regimes only in \cite{GV}. All the mentioned works  do rely on the MFG structure of \eqref{hjb}-\eqref{fkp}. Here, we just exploit standard regularizing properties of the diffusion, and propose a general existence result for \eqref{hjb}-\eqref{fkp} that requires very mild local (regularity) assumptions on the nonlinearity $H$. The key tool is the standard contraction mapping theorem, that has already been explored in the MFG setting in \cite{CT} and \cite{Amb} in more particular cases.

Here, $H(x, t, p, m) = m^\alpha H_1(p / m^\alpha) - f(x, t, m)$ where $H_1$ is the Legendre transform of $L_1$, so
\[
\begin{split}
F(x,t, u, m, Du, Dm)& = m^\alpha H_1\left(\frac{Du } {m^\alpha} \right) - f(x, t, m), \\
G(x,t, u, m, Du, D^2 u, Dm)& = -(\partial_{x_i x_j} A_{ij})m - 2 (\partial_{x_i} A_{ij}) m_{x_j} - D_p H_1 \cdot D m \\
& -m^{1-\alpha} (H_1)_{p_i p_j} u_{x_i x_j} +  {\alpha}{m^{-\alpha}}  (H_1)_{p_i p_j} u_{x_j} m_{x_i}.
\end{split}
\]
The MFG system then takes the form
\begin{equation}\label{mfgcong}
\begin{cases}
-u_t - A_{ij} u_{ij} + m^\alpha H_1(Du / m^\alpha)  = f(x, t, m) & \text{in $\mathbb{T}^N\times (0,T)$}, \\
m_t - \partial_{ij} (A_{ij} m) - {\rm div} (m D_p H_1(Du / m^\alpha) ) = 0 & \text{in $\mathbb{T}^N\times (0,T)$}, \\
u(x,T)=h[m(T)](x), \quad m(x,0)=m_0(x) & \text{in $\mathbb{T}^N$}.
\end{cases}
\end{equation}

A corollary of Theorem \ref{shortmfg} thus reads
\begin{corollary}
Suppose that $A \in C^{2}(Q_T)$, $h$ satisfies (A4), $m_0$ satisfies (A5) and
\begin{itemize}
\item $f$ is continuous with respect to $x,t,m$ and locally Lipschitz continuous with respect to $m$,
\item $H_1$ is continuous and has second derivatives that are locally Lipschitz continuous.
\end{itemize}
Then, there exists $\overline T>0$ such that for all $T \in (0, \overline T]$ the system \eqref{mfgcong} admits a solution $u,m\in W^{2,1}_q(Q_T)$ with $q>N+2$.
\end{corollary}

\appendix
\section{Appendix}\label{appe}
In this final appendix we prove Proposition \ref{w2pest} of Section \ref{ass}.
%
%
\begin{proof}[Proposition \ref{w2pest}]
We write the proof for the
existence of a solution in the class $C^{2+\alpha, 1+\alpha/2}(Q_T)$ and for the estimate \eqref{320}. In a similar way one obtains the existence in $W^{2,1}_q(Q_T)$ and the proof of \eqref{342}.\\
Recall that the problem on $\mathbb{T}^N\times [0,T]$  is equivalent to the same problem with $1$-periodic data in the $x$-variable in $\re^N\times [0,T]$, namely with all the data satisfying $w(x+z,t)=w(x,t)$ for all $z\in\z^N$.
As far as the existence of a smooth solution of problem \eqref{linapp} is concerned, 
it is sufficient to apply Theorem 5.1 p.320 of \cite{LSU}. Since the solution of such a Cauchy problem is unique, it must be periodic in the $x$-variable.
Now we prove estimate \eqref{320}. 
Let $R_1^N:=[-1,1]^N$ and $R_2^N := [-2,2]^N$.
Clearly \begin{equation}\label{inclu}
[0,1]^N \subset R_1^N\subset R_2^N \subset \re^N
\end{equation} 
and $dist(R_1^N, C(R_2^N))=1$.



We take advantage of local parabolic estimates, which allow us to get an a priori estimate regardless of the lateral boundary conditions which are unknown for us.\\ 
In particular,
using the local estimate (10.5) p. 352 of \cite{LSU} with $\Omega'=R_1^N$ and 
$\Omega'':={R_2^N}$, (note that in our case $S''$ is empty) we have
\begin{equation}\label{locest}
  |u|^{(\alpha+2)}_{R_1^N\times [0,T^*]}\leq C_1\left(|f|^{(\alpha)}_{R_2^N\times [0,T^*]}+ 
  |u_0|^{(\alpha+2)}_{R_2^N}\right)+ C_2 |u|_{R_2^N\times [0,T^*]},
 \end{equation}
 where $T^*<T$, $C_1$ and $C_2$ depend on $N$, $T$, and the 
 modulus of H\"older continuity of the coefficients of the operator. It is now crucial to observe that H\" older norms on $R_1^N\times [0,T^*]$, $R_2^N\times [0,T^*]$ and $\mathbb T^N\times [0,T^*]$ coincide by periodicity of $u$, $f$, $u_0$ in the $x$-variable and the inclusions \eqref{inclu}.
 Hence,
\begin{equation}\label{1}
  |u|^{(\alpha+2)}_{\mathbb T^N\times [0,T^*]}\leq C_1\left(|f|^{(\alpha)}_{\mathbb T^N\times [0,T^*]}+ 
  |u_0|^{(\alpha+2)}_{\mathbb T^N}\right)+ C_2(|u_0|_{\mathbb T^N}+ T^*|u_t|_{\mathbb T^N\times [0,T^*]}).
\end{equation}
Taking $T^*$ sufficiently small we can write 
\begin{equation}\label{step1}
  |u|^{(\alpha+2)}_{\mathbb T^N\times [0,T^*]}\leq C_3\left(|f|^{(\alpha)}_{\mathbb T^N\times [0,T^*]}+ 
  |u_0|^{(\alpha+2)}_{\mathbb T^N}\right),
 \end{equation}
 where $C_3$ depends on the coefficients of the equation, on $N$, $T$, $\alpha$ and $T^*$ does not depend on $u_0$. We can iterate the estimate \eqref{step1}
 to cover all the interval $[0,T]$ in $[\frac{T}{T^*}]+1$ steps, thus obtaining \eqref{320}.
 
 The proof of \eqref{342} is completely analogous. One has to exploit the local estimate in $W^{2,1}_p$ of \cite{LSU}, eq. (10.12), p. 355. Note that since $R_1^N$ and $R_2^N$ consist of finite copies of $[0,1]^N$, norms on $W_p^{2m, m}(R_i^N \times (0,T))$, $i=1,2$, are multiples (depending on $N$) of $W_p^{2m, m}(\mathbb T^N \times (0, T))$.
 \\

\end{proof}

\noindent{\bf Acknowledgments.} The first and third authors are members of GNAMPA-INdAM, and were partially supported by the research project of the University of Padova ``Mean-Field Games and Nonlinear PDEs" and by the Fondazione CaRiPaRo Project ``Nonlinear Partial Differential Equations: Asymptotic Problems and Mean-Field Games".


\end{document}